\documentclass[12pt,reqno,english]{smfart}

\usepackage{amsmath,amssymb,smfthm,stmaryrd,epic,enumerate,dsfont}
\usepackage[english]{babel}
\usepackage[all]{xy}
\usepackage{csquotes}
\usepackage{mathabx}

\usepackage{xcolor}

\usepackage{mathrsfs}
\usepackage[active]{srcltx}
\usepackage{graphicx}

\NumberTheoremsIn{subsection}\SwapTheoremNumbers

\begin{document}

\newcommand{\REMARK}[1]{\marginpar{\tiny #1}}
\newtheorem{thm}{Theorem}[section]
\newtheorem{defin}[thm]{Definition}
\newtheorem{Remark}[thm]{Remark}
\numberwithin{equation}{subsection}
\newtheorem{coro}[thm]{Corollary}
\newtheorem{prop}[thm]{Proposition}
\newtheorem{lemma}[thm]{Lemma}

\newtheorem{notas}[thm]{Notations}
\newtheorem{nota}[thm]{Notation}
\newtheorem*{thm*}{Theorem}
\newtheorem*{prop*}{Proposition}
\newtheorem{conje}[thm]{Conjecture}
\newtheorem*{conj*}{Conjecture}
\newtheorem{hyp}[thm]{Hypothesis}

\def\Tm{{\mathbb T}}
\def\Um{{\mathbb U}}
\def\Am{{\mathbb A}}
\def\Fm{{\mathbb F}}
\def\Mm{{\mathbb M}}
\def\Nm{{\mathbb N}}
\def\Pm{{\mathbb P}}
\def\Qm{{\mathbb Q}}
\def\Zm{{\mathbb Z}}
\def\Dm{{\mathbb D}}
\def\Cm{{\mathbb C}}
\def\Rm{{\mathbb R}}
\def\Gm{{\mathbb G}}
\def\Lm{{\mathbb L}}
\def\Km{{\mathbb K}}
\def\Om{{\mathbb O}}
\def\Em{{\mathbb E}}
\def\Xm{{\mathbb X}}

\def\BC{{\mathcal B}}
\def\QC{{\mathcal Q}}
\def\TC{{\mathcal T}}
\def\ZC{{\mathcal Z}}
\def\AC{{\mathcal A}}
\def\CC{{\mathcal C}}
\def\DC{{\mathcal D}}
\def\EC{{\mathcal E}}
\def\FC{{\mathcal F}}
\def\GC{{\mathcal G}}
\def\HC{{\mathcal H}}
\def\IC{{\mathcal I}}
\def\JC{{\mathcal J}}
\def\KC{{\mathcal K}}
\def\LC{{\mathcal L}}
\def\MC{{\mathcal M}}
\def\NC{{\mathcal N}}
\def\OC{{\mathcal O}}
\def\PC{{\mathcal P}}
\def\UC{{\mathcal U}}
\def\VC{{\mathcal V}}
\def\AC{{\mathcal A}}
\def\SC{{\mathcal S}}
\def\RC{{\mathcal R}}

\def\BF{{\mathfrak B}}
\def\AF{{\mathfrak A}}
\def\GF{{\mathfrak G}}
\def\EF{{\mathfrak E}}
\def\CF{{\mathfrak C}}
\def\DF{{\mathfrak D}}
\def\JF{{\mathfrak J}}
\def\LF{{\mathfrak L}}
\def\MF{{\mathfrak M}}
\def\NF{{\mathfrak N}}
\def\XF{{\mathfrak X}}
\def\UF{{\mathfrak U}}
\def\KF{{\mathfrak K}}
\def\FF{{\mathfrak F}}

\def \longmapright#1{\smash{\mathop{\longrightarrow}\limits^{#1}}}
\def \mapright#1{\smash{\mathop{\rightarrow}\limits^{#1}}}
\def \lexp#1#2{\kern \scriptspace \vphantom{#2}^{#1}\kern-\scriptspace#2}
\def \linf#1#2{\kern \scriptspace \vphantom{#2}_{#1}\kern-\scriptspace#2}
\def \linexp#1#2#3 {\kern \scriptspace{#3}_{#1}^{#2} \kern-\scriptspace #3}

\def \Sel {{\mathop{\mathrm{Sel}}\nolimits}}
\def \Ext{\mathop{\mathrm{Ext}}\nolimits}
\def \ad{\mathop{\mathrm{ad}}\nolimits}
\def \sh{\mathop{\mathrm{Sh}}\nolimits}
\def \irr{\mathop{\mathrm{Irr}}\nolimits}
\def \FH{\mathop{\mathrm{FH}}\nolimits}
\def \FPH{\mathop{\mathrm{FPH}}\nolimits}
\def \coh{\mathop{\mathrm{Coh}}\nolimits}
\def \res{\mathop{\mathrm{Res}}\nolimits}
\def \op{\mathop{\mathrm{op}}\nolimits}
\def \rec {\mathop{\mathrm{rec}}\nolimits}
\def \art{\mathop{\mathrm{Art}}\nolimits}
\def \vol {\mathop{\mathrm{vol}}\nolimits}
\def \cusp {\mathop{\mathrm{Cusp}}\nolimits}
\def \scusp {\mathop{\mathrm{Scusp}}\nolimits}
\def \Iw {\mathop{\mathrm{Iw}}\nolimits}
\def \JL {\mathop{\mathrm{JL}}\nolimits}
\def \speh {\mathop{\mathrm{Speh}}\nolimits}
\def \isom {\mathop{\mathrm{Isom}}\nolimits}
\def \Vect {\mathop{\mathrm{Vect}}\nolimits}
\def \groth {\mathop{\mathrm{Groth}}\nolimits}
\def \hom {\mathop{\mathrm{Hom}}\nolimits}
\def \deg {\mathop{\mathrm{deg}}\nolimits}
\def \val {\mathop{\mathrm{val}}\nolimits}
\def \det {\mathop{\mathrm{det}}\nolimits}
\def \rep {\mathop{\mathrm{Rep}}\nolimits}
\def \spec {\mathop{\mathrm{Spec}}\nolimits}
\def \fr {\mathop{\mathrm{Fr}}\nolimits}
\def \frob {\mathop{\mathrm{Frob}}\nolimits}
\def \ker {\mathop{\mathrm{Ker}}\nolimits}
\def \im {\mathop{\mathrm{Im}}\nolimits}
\def \Red {\mathop{\mathrm{Red}}\nolimits}
\def \red {\mathop{\mathrm{red}}\nolimits}
\def \aut {\mathop{\mathrm{Aut}}\nolimits}
\def \diag {\mathop{\mathrm{diag}}\nolimits}
\def \spf {\mathop{\mathrm{Spf}}\nolimits}
\def \Def {\mathop{\mathrm{Def}}\nolimits}
\def \twist {\mathop{\mathrm{Twist}}\nolimits}
\def \supp {\mathop{\mathrm{Supp}}\nolimits}
\def \Id {{\mathop{\mathrm{Id}}\nolimits}}
\def \lie {{\mathop{\mathrm{Lie~}}\nolimits}}
\def \Ind{\mathop{\mathrm{Ind}}\nolimits}
\def \ind {\mathop{\mathrm{ind}}\nolimits}
\def \bad {\mathop{\mathrm{Bad}}\nolimits}
\def \top {\mathop{\mathrm{Top}}\nolimits}
\def \ker {\mathop{\mathrm{Ker}}\nolimits}
\def \coker {\mathop{\mathrm{Coker}}\nolimits}
\def \gal {{\mathop{\mathrm{Gal}}\nolimits}}
\def \Nr {{\mathop{\mathrm{Nr}}\nolimits}}
\def \rn {{\mathop{\mathrm{rn}}\nolimits}}
\def \tr {{\mathop{\mathrm{Tr~}}\nolimits}}
\def \Sp {{\mathop{\mathrm{Sp}}\nolimits}}
\def \st {{\mathop{\mathrm{St}}\nolimits}}
\def \sp{{\mathop{\mathrm{Sp}}\nolimits}}
\def \perv{\mathop{\mathrm{Perv}}\nolimits}
\def \tor {{\mathop{\mathrm{Tor}}\nolimits}}
\def \gr {{\mathop{\mathrm{gr}}\nolimits}}
\def \nilp {{\mathop{\mathrm{Nilp}}\nolimits}}
\def \obj {{\mathop{\mathrm{Obj}}\nolimits}}
\def \spl {{\mathop{\mathrm{Spl}}\nolimits}}
\def \unr {{\mathop{\mathrm{Unr}}\nolimits}}
\def \alg {{\mathop{\mathrm{Alg}}\nolimits}}
\def \grr {{\mathop{\mathrm{grr}}\nolimits}}
\def \cogr {{\mathop{\mathrm{cogr}}\nolimits}}
\def \coFil {{\mathop{\mathrm{coFil}}\nolimits}}

\def \rem{{\noindent\textit{Remark.~}}}
\def \rems{{\noindent\textit{Remarks:~}}}
\def \ext {{\mathop{\mathrm{Ext}}\nolimits}}
\def \End {{\mathop{\mathrm{End}}\nolimits}}

\def\semi{\mathrel{>\!\!\!\triangleleft}}
\let \DS=\displaystyle
\def\HT{{\mathop{\mathcal{HT}}\nolimits}}

\def \hi{\HC}
\newcommand*{\tarrow}{\relbar\joinrel\mid\joinrel\twoheadrightarrow}
\newcommand*{\harrow}{\lhook\joinrel\relbar\joinrel\mid\joinrel\rightarrow}
\newcommand*{\rarrow}{\relbar\joinrel\mid\joinrel\rightarrow}
\def \coim {{\mathop{\mathrm{Coim}}\nolimits}}
\def \can {{\mathop{\mathrm{can}}\nolimits}}
\def\LFF{{\mathscr L}}

\setcounter{secnumdepth}{3} \setcounter{tocdepth}{3}

\def \Fil{\mathop{\mathrm{Fil}}\nolimits}
\def \CoFil{\mathop{\mathrm{CoFil}}\nolimits}
\def \Fill{\mathop{\mathrm{Fill}}\nolimits}
\def \CoFill{\mathop{\mathrm{CoFill}}\nolimits}
\def\SF{{\mathfrak S}}
\def\PF{{\mathfrak P}}
\def \EFil{\mathop{\mathrm{EFil}}\nolimits}
\def \EFill{\mathop{\mathrm{EFill}}\nolimits}
\def \FP{\mathop{\mathrm{FP}}\nolimits}

\let \longto=\longrightarrow
\let \oo=\infty

\let \d=\delta
\let \k=\kappa

\renewcommand{\theequation}{\arabic{section}.\arabic{thm}}
\newcommand{\marque}{\addtocounter{thm}{1}
{\smallskip \noindent \textit{\thethm}~---~}}

\renewcommand\atop[2]{\ensuremath{\genfrac..{0pt}{1}{#1}{#2}}}

\newcommand\atopp[2]{\genfrac{}{}{0pt}{}{#1}{#2}}

\title[Automorphic torsion congruences]{Automorphic congruences between torsion cohomological classes}


\author{Boyer Pascal}
\email{boyer@math.univ-paris13.fr}
\address{Universit\'e Paris 13, Sorbonne Paris Nord \\
LAGA, CNRS, UMR 7539\\ 
F-93430, Villetaneuse (France) \\
Coloss: ANR-19-PRC}

\frontmatter

\begin{abstract}
For two representations of some local division algebra, congruent modulo $l$, 
giving rise to two Harris-Taylor local systems on the corresponding Newton strata of the
special fiber of a KHT Shimura varieties, we prove that the $l$-torsion of each of their 
cohomology groups with compact supports are isomorphic, or equivalently the free quotients 
of each of the cohomology
groups are congruent modulo $l$. We then deduce the 
construction of accurate non tempered automorphic congruences for a similitude
group $G/\Qm$ with signature $(1,d-1)$. 
\end{abstract}



\maketitle

\pagestyle{headings} \pagenumbering{arabic}

%
%

\maketitle


\tableofcontents

\section{Introduction}

The first appearance of automorphic congruences can be traced back to
Ramanujan's works on the $\tau$-functions. Now existence and
construction of higher dimensional automorphic congruences play an
essential role in particular to the Langlands program. 

One possible geometric approach
is to look at the cohomology groups of some $\overline \Zm_l$-local system 
on the special fiber, at a place $v$ of the reflex field $F$, of a Shimura variety associated to 
a reductive group $G/\Qm$,
whose free quotients are expected to be of automorphic
nature. The idea is then to take two such local systems whose modulo $l$ reductions
are isomorphic so that the modulo $l$ reduction of the alternated sum of their
cohomology groups are equal in the corresponding Grothendieck group of
$\overline \Fm_l$-representations of $G(\Am)$, cf. lemma 4.1.5 of \cite{boyer-aif}.
If one expect to construct accurate automorphic congruences we need at least to deal
with individual groups of cohomology and not just the alternated sum.
We then face two main problems.
\begin{itemize}
\item[(a)] The torsion may interfere and prevent us to relate the modulo $l$ reduction
of the free quotients of the cohomology groups of our two $\overline \Zm_l$-local
systems.

\item[(b)] Even if we can manage about the torsion, the 
$\overline \Qm_l$-cohomology of our local systems may involve different sorts
of automorphic representations and we would only be able to construct
unaccurate automorphic congruences.
\end{itemize}

\noindent - Usually to overcome (a), one localizes at a non pseudo Eisenstein ideal 
but we then loose degenerate automorphic representations which is not satisfactory. 

\noindent - For problem (b), over $\overline \Qm_l$,  working
with intermediate extensions of local systems, allows, using purity,
to separate the contributions of each cohomology groups. Meanwhile we still face the same
issue with various automorphic representations contributing to the same cohomology group.
One idea is then to consider various local systems and try to cross-check the informations.

When considering $\overline \Zm_l$-intermediate extensions, before taking
their modulo $l$ reduction, appears the problem
that they may be many $\overline \Zm_l$-intermediate extensions whose
modulo $l$ reductions do not coincide: if we need to use duality this might cause
an issue. This is why
in \cite{boyer-aif} we only deal with extension by zero of our local systems and
then consider only the alternated sum of their cohomology groups.
In \cite{boyer-duke} we solve this problem, i.e. we are able to explain the 
difference\footnote{more precisely we understand the $l$-torsion of the quotient between
two intermediate extensions of a Harris-Taylor local system}
between two different intermediate
extensions for
Harris-Taylor local systems living on Newton strata of KHT-Shimura varieties.

\medskip

In this article, in the banal case\footnote{i.e. when the order of $q_v$, the cardinal of the residue filed of $F$ at $v$, modulo $l$ is strictly greater than $d$}, playing with all the Harris-Taylor local systems associated 
to one cuspidal representation, we are able to
identify inside one cohomology group of two congruent Harris-Taylor local
systems, both the torsion and the various contributions of the automorphic
representations. More precisely
\begin{itemize}
\item[(a)] we first prove the conjecture 5.10
of \cite{boyer-aif}, saying that the modulo $l$ reduction of the torsion
submodule of the cohomology groups of both the intermediate extensions and extensions by zero,
of Harris-Taylor local systems,
only depends on the modulo $l$ reduction of the local system we started with. 

\item[(b)] Secondly, thanks to the computations of \cite{boyer-aif} recalled in
\S \ref{para-Hi},
we are able to identify, inside one particular
cohomology group of two congruent Harris-Taylor local systems, the contributions
of the automorphic representations
according to the shape of their local component at the place considered.
\end{itemize}

As an application, we are able to prove the following statement about
automorphic congruences, cf. theorem \ref{thm-main}. We start with an irreducible automorphic
representation $\Pi$ of $G(\Am)$ with weight $\xi$, an irreducible algebraic representation of $G(\Qm)$, which locally at some prime number\footnote{$l$ is still supposed to be banal} $p \neq l$ is,
with the notations of the next section, of the form
$$\Pi_p \simeq \Pi_p^v \otimes \speh_s \bigl ( \st_{t_1}(\pi_{v,1}) \times \cdots \times \st_{t_r}(\pi_{v,r}) \bigr ),$$ 
for some $s \geq 1$, where, cf. the convention of \S \ref{para-KHT},
$v| p$ is a fixed split place of the CM field $F$ defining $G$  and where
$\pi_{v,1},\cdots,\pi_{v,r}$ are irreducible cuspidal representations. We then 
consider both a weight $\xi'$ congruent to $\xi$ modulo $l$ and
some local congruence $\pi'_{v,1}$ of $\pi_{v,1}$ which we suppose to be 
supercuspidal modulo $l$. We then prove, cf. the introduction 
of \S \ref{para-congruence}, the existence of an irreducible automorphic
representation $\Pi'$ of $G(\Am)$ of weight $\xi'$ such that
\begin{itemize}
\item locally $\Pi'$ at the place $v$ is isomorphic to 
$$\speh_s \Bigl (\st_{t'_1}(\pi'_{v,1}) \times \cdots \times \st_{t'_{r'}}(\pi'_{v,r'}) \Bigr )$$ 
with $t'_1=t_1$ and where for $i=2,\cdots,r'$, the $\pi'_{v,i}$ are
irreducible cuspidal representations,

\item globally $\Pi'$ share with $\Pi$ the same level outside $v$ and is weakly
congruent to $\Pi$ in the sense that, modulo an uniformizer $\varpi_L$ of some finite extension
$L/\Qm_l$, it shares the same Satake parameters
at the unramified places.
\end{itemize}
More precisely theorem \ref{thm-main} gives a quantitative version in terms of an
equality between multiplicities in the space of automorphic forms.

\section{Notations about representations of $GL_n(K)$}

\label{para-rap-rep}

We fix a finite extension $K/\Qm_p$ with residue field $\Fm_q$.
We denote by $|-|$ its absolute value and $\val$ for the valuation. We fix a prime $l$ distinct from $p$.
For a representation $\pi$ of $GL_d(K)$ and $n \in \frac{1}{2} \Zm$, set 
$$\pi \{ n \}:= \pi \otimes q^{-n \val \circ \det}.$$

\begin{notas} \label{nota-ind}
For $\pi_1$ and $\pi_2$ representations of respectively $GL_{n_1}(K)$ and
$GL_{n_2}(K)$, we will denote by
$$\pi_1 \times \pi_2:=\ind_{P_{n_1,n_1+n_2}(K)}^{GL_{n_1+n_2}(K)}
\pi_1 \{ \frac{n_2}{2} \} \otimes \pi_2 \{-\frac{n_1}{2} \},$$
the normalized parabolic induced representation where for any sequence
$\underline r=(0< r_1 < r_2 < \cdots < r_k=d)$, we write $P_{\underline r}$ for 
the standard parabolic subgroup of $GL_d$ with Levi
$$GL_{r_1} \times GL_{r_2-r_1} \times \cdots \times GL_{r_k-r_{k-1}}.$$ 
\end{notas}

Recall that a representation
$\varrho$ of $GL_d(K)$ with coefficients either in $\overline \Qm_l$ or $\overline \Fm_l$,
is called \emph{cuspidal} (resp. \emph{supercuspidal})
if it is not a subspace (resp. subquotient) of a proper parabolic induced representation.
When the field of coefficients is of characteristic zero then these two notions coincides,
but this is no more true for $\overline \Fm_l$.

\begin{defin} \label{defi-rep} 
(see \cite{zelevinski2} \S 9 and \cite{boyer-compositio} \S 1.4)
Let $g$ be a divisor of $d=sg$ and $\pi$ an irreducible cuspidal 
$\overline \Qm_l$-representation of $GL_g(K)$. 
The induced representation
$$\pi\{ \frac{1-s}{2} \} \times \pi \{ \frac{3-s}{2} \} \times \cdots \times \pi \{ \frac{s-1}{2} \}$$ 
holds a unique irreducible quotient (resp. subspace) denoted $\st_s(\pi)$ (resp.
$\speh_s(\pi)$); it is called a generalized Steinberg (resp. Speh) representation.
\end{defin}

Any generic irreducible representation $\Pi$ of $GL_d(K)$ is isomorphic to
$\st_{t_1}(\pi_1) \times \cdots \times \st_{t_r}(\pi_r)$
where for $i=1,\cdots, r$, the $\pi_i$ are irreducible cuspidal representations of 
$GL_{g_i}(K)$ and $t_i \geq 1$ are such that $\sum_{i=1}^r t_ig_i=d$. 
For $\Pi$ an irreducible generic representation and $s \geq 1$, we denote by
$$\speh_s(\Pi)=\speh_s \bigl ( \st_{t_1}(\pi_1) \bigr ) \times \cdots \times \speh_s \bigl (\st_{t_r}(\pi_r) \bigr )$$
the Langlands quotient of the parabolic induced representation 
$\Pi \{ \frac{1-s}{2} \} \times 
\Pi \{ \frac{3-s}{2} \} \times \cdots \times \Pi \{ \frac{s-1}{2} \}$. 
In terms of the local Langlands correspondence, if $\sigma$ is the representation of 
the absolute Galois group
$\gal(\bar F/F)$ of $F$, associated to $\Pi$, then
$\speh_s(\Pi)$ corresponds to 
$\sigma(\frac{1-s}{2}) \oplus \cdots \oplus \sigma(\frac{s-1}{2})$ where
$\sigma(\frac{1-s+2k}{2})$ correspond to $\Pi \{ \frac{1-s+2k}{2} \}$ by the global
Langlands correspondance.

\begin{defin} Let $D_{K,d}$ be the central division algebra over $K$ with invariant 
$1/d$ and with maximal order denoted by $\DC_{K,d}$. 
\end{defin}

The local Jacquet-Langlands correspondance is a bijection $\JL$ 
between the set of equivalence classes of irreducible admissible representations
of $D_{K,d}^\times$ and the one of irreducible admissible essentially square
integrable representations of $GL_{d}(K)$.

\begin{nota} 
For $\pi$ a cuspidal irreducible $\bar \Qm_l$-representation of $GL_g(K)$ 
and for $t \geq 1$, we then denote by
$\pi[t]_D$ the representation $\JL^{-1}(\st_t(\pi))^\vee$ of $D_{K,tg}^\times$,
where the symbol $\vee$ stands for contragredient.
\end{nota}

\section{KHT-Shimura varieties and Harris-Taylor local systems}
\label{para-KHT}

Let $F=F^+ E$ be a CM field where $E/\Qm$ is quadratic imaginary and $F^+/\Qm$
totally real with a fixed real embedding $\tau:F^+ \hookrightarrow \Rm$. 
We denote by $c$ the non trivial element of $\gal(E/\Qm)$. For a place $v$ of $F$,
we will denote by
\begin{itemize}
\item $F_v$ the completion of $F$ at $v$,

\item $\OC_v$ the ring of integers of $F_v$ with maximal ideal $\MC_v$,

\item $\varpi_v$ a uniformizer,

\item $q_v$ the order of the residual field $\kappa(v)=\OC_v/(\varpi_v)$.
\end{itemize}
Let $B$ be a division algebra with center $F$, of dimension $d^2$ such that at every place $x$ of $F$,
either $B_x$ is split or a local division algebra and suppose $B$ provided with an involution of
second kind $*$ such that $*_{|F}$ is the complex conjugation. For any
$\beta \in B^{*=-1}=\{ x \in B: x^*=-x \}$, denote by $\sharp_\beta$ the involution $x \mapsto x^{\sharp_\beta}=\beta x^*
\beta^{-1}$ and $G/\Qm$ the group of similitudes, denoted $G_\tau$ in \cite{h-t}, defined for every
$\Qm$-algebra $R$ by 
$$
G(R)  \simeq   \{ (\lambda,g) \in R^\times \times (B^{op} \otimes_\Qm R)^\times  \hbox{ such that } 
gg^{\sharp_\beta}=\lambda \}
$$
with $B^{op}=B \otimes_{F,c} F$. 
If $x$ is a place of $\Qm$ split $x=yy^c$ in $E$ then 
\addtocounter{thm}{1}
\begin{equation} \label{eq-facteur-v}
G(\Qm_x) \simeq (B_y^{op})^\times \times \Qm_x^\times \simeq \Qm_x^\times \times
\prod_{z_i} (B_{z_i}^{op})^\times,
\end{equation}
where, identifying places of $F^+$ over $x$ with places of $F$ over $y$,
$x=\prod_i z_i$ in $F^+$.

\noindent \textbf{Convention}: for $x=yy^c$ a place of $\Qm$ split in $E$ and $z$ a place of $F$ over $y$
as before, we shall make throughout the text, the following abuse of notation by denoting 
$G(F_z)$ in place of the factor $(B_z^{op})^\times$ in the formula (\ref{eq-facteur-v}).

In \cite{h-t}, the author justifies the existence of $G$ like before such that moreover
\begin{itemize}
\item if $x$ is a place of $\Qm$ non split in $E$ then $G(\Qm_x)$ is quasi split;

\item the invariants of $G(\Rm)$ are $(1,d-1)$ for the embedding $\tau$ and $(0,d)$ for the others.
\end{itemize}

\begin{nota}
We denote by $\IC$ the set of open compact subgroups 
small enough of $G(\Am^\oo)$.
For $I \in \IC$,  $\sh_{I,\eta} \longrightarrow \spec F$ is the associated
generic fiber of the Shimura variety said of Kottwitz-Harris-Taylor type.
\end{nota}

\begin{defin} \label{defi-spl}
Define $\spl$ as the set of  places $v$ of $F$ such that $p_v:=v_{|\Qm} \neq l$ is split in $E$ and
$B_v^\times \simeq GL_d(F_v)$.  For each $I \in \IC$, write
$\spl(I)$ the subset of $\spl$ of places which do not divide the level $I$.
\end{defin}

As in  \cite{h-t} bottom of page 90, a compact open subgroup $U$ of $G(\Am^\oo)$ is 
said \emph{small enough}
if there exists a place $x$ such that the projection from $U^v$ to $G(\Qm_x)$ does 
not contain any element of finite order except identity.

In the sequel, $v$ will denote a fixed place of $F$ in $\spl$. For such a place $v$ 
the scheme $\sh_{I,\eta}$ has a projective model $\sh_{I,v}$ over $\spec \OC_v$
with geometric special fiber $\sh_{I,\bar s_v}$. For $I$ going through $\IC$, the projective system 
$(\sh_{I,v})_{I\in \IC}$ 
is naturally equipped with an action of $G(\Am^\oo) \times \Zm$ such that 
$w_v$ in the Weil group $W_v$ of $F_v$ acts by $-\deg (w_v) \in \Zm$,
where $\deg=\val \circ \art^{-1}$ and $\art^{-1}:W_v^{ab} \simeq F_v^\times$ is 
the Artin isomorphism which sends geometric Frobenius to uniformizers.

\begin{notas} (see \cite{boyer-invent2} \S 1.3)
For $I \in \IC$, the Newton stratification of the geometric special fiber 
$\sh_{I,\bar s_v}$ is denoted by
$$\sh_{I,\bar s_v}=:\sh^{\geq 1}_{I,\bar s_v} \supset \sh^{\geq 2}_{I,\bar s_v} 
\supset \cdots \supset \sh^{\geq d}_{I,\bar s_v}$$
where $\sh^{=h}_{I,\bar s_v}:=\sh^{\geq h}_{I,\bar s_v} - \sh^{\geq h+1}_{I,\bar s_v}$ is 
an affine scheme\footnote{see for example \cite{ito2}}, smooth of pure dimension 
$d-h$ built up by the geometric 
points whose connected part of its Barsotti-Tate group is of rank $h$.
For each $1 \leq h <d$, write
$$i_{h}:\sh^{\geq h}_{I,\bar s_v} \hookrightarrow \sh^{\geq 1}_{I,\bar s_v}, \quad
j^{\geq h}: \sh^{=h}_{I,\bar s_v} \hookrightarrow \sh^{\geq h}_{I,\bar s_v},$$
and $j^{=h}=i_h \circ j^{\geq h}$.
\end{notas}

Consider now the ideals $I^v(n):=I^vK_v(n)$ where
$$K_v(n):=\ker(GL_d(\OC_v) \twoheadrightarrow GL_d(\OC_v/\MC_v^n)).$$
Recall then that $\sh_{I^v(n),\bar s_v}^{=h}$ is geometrically induced
under the action of the parabolic subgroup $P_{h,d}(\OC_v/\MC_v^n)$. 
Concretely this 
means that there exists a closed subscheme 
$\sh_{I^v(n),\bar s_v,\overline{1_{h}}}^{=h}$ stabilized by the Hecke 
action of $P_{h,d}(F_v)$ and such that
$$\sh_{I^v(n),\bar s_v}^{=h} = \sh_{I^v(n),\bar s_v,\overline{1_{h}}}^{=h} 
\times_{P_{h,d}(\OC_v/\MC_v^n)} GL_d(\OC_v/\MC_v^n),$$
meaning that $\sh_{I^v(n),\bar s_v}^{=h} $ is the disjoint union of copies of
$\sh_{I^v(n),\bar s_v,\overline{1_{h}}}^{=h}$ indexed by 
$GL_d(\OC_v/\MC_v^n)/P_{h,d}(\OC_v/\MC_v^n)$ and 
exchanged by the action of
$GL_d(\OC_v/\MC_v^n)$.

\begin{notas} 
For $1 \leq t \leq s_g:=\lfloor d/g \rfloor$, let $\Pi_t$ be any representation of 
 $GL_{d-tg}(F_v)$. We then denote by
$$\widetilde{HT}_1(\pi_v,\Pi_t):=\LC(\pi_v[t]_D)_{\overline{1_{tg}}} 
\otimes \Pi_t \otimes \Xi^{\frac{tg-d}{2}}$$ 
the Harris-Taylor local system on the Newton stratum $\sh^{=tg}_{I,\bar s_v,\overline{1_{tg}}}$ where 
\begin{itemize}
\item $\LC(\pi_v[t]_D)_{\overline{1_{tg}}}$ is defined thanks\footnote{It is the local
system associated to the $(\pi_v[t]_D)_{|\DC_{v,tg}^\times}$ by the Igusa variety
which is Galois over the Newton stratum with Galois group $\DC_{v,tg}^\times$,
cf. \cite{boyer-invent2} \S 1.4.}
to Igusa varieties attached to the representation $\pi_v[t]_D$,

\item $\Xi:\frac{1}{2} \Zm \longrightarrow \overline \Zm_l^\times$ is defined by 
$\Xi(\frac{1}{2})=q^{1/2}$.
\end{itemize}
We also introduce the induced version
$$\widetilde{HT}(\pi_v,\Pi_t):=\Bigl ( \LC(\pi_v[t]_D)_{\overline{1_{tg}}} 
\otimes \Pi_t \otimes \Xi^{\frac{tg-d}{2}} \Bigr) \times_{P_{tg,d}(F_v)} GL_d(F_v),$$
where the unipotent radical of $P_{tg,d}(F_v)$ acts trivially and the action of
$$(g^{\oo,v},\left ( \begin{array}{cc} g_v^c & * \\ 0 & g_v^{et} \end{array} \right ),\sigma_v) 
\in G(\Am^{\oo,v}) \times P_{tg,d}(F_v) \times W_v$$ 
is given
\begin{itemize}
\item by the action of $g_v^c$ on $\Pi_t$ and 
$\deg(\sigma_v) \in \Zm$ on $ \Xi^{\frac{tg-d}{2}}$, and

\item the action of $(g^{\oo,v},g_v^{et},\val(\det g_v^c)-\deg \sigma_v)
\in G(\Am^{\oo,v}) \times GL_{d-tg}(F_v) \times \Zm$ on $\LC_{\overline \Qm_l}
(\pi_v[t]_D)_{\overline{1_{tg}}} \otimes \Xi^{\frac{tg-d}{2}}$.
\end{itemize}
We also introduce
$$HT(\pi_v,\Pi_t)_{\overline{1_{tg}}}:=\widetilde{HT}(\pi_v,\Pi_t)_{\overline{1_{tg}}}[d-tg],$$
and the perverse sheaf
$$P(t,\pi_v)_{\overline{1_{tg}}}:=j^{=tg}_{1,!*} HT(\pi_v,\st_t(\pi_v))_{\overline{1_{tg}}} 
\otimes \Lm(\pi_v),$$
and their induced version, $HT(\pi_v,\Pi_t)$ and $P(t,\pi_v)$, where 
$$j^{=h}=i^h \circ j^{\geq h}:\sh^{=h}_{I,\bar s_v} \hookrightarrow
\sh^{\geq h}_{I,\bar s_v} \hookrightarrow \sh_{I,\bar s_v}$$ 
and $\Lm^\vee$, the dual of $\Lm$, is the local Langlands correspondence.
\end{notas}

For $\overline \Qm_l$ or $\overline \Fm_l$ coefficients, we will mention it in the index of the local
system, as for example $HT_{\overline \Qm_l}(\pi_v,\Pi_t)$ or  $HT_{\overline \Fm_l}(\pi_v,\Pi_t)$.
More precisely if $L$ is a finite extension of $\Qm_l$ so that both $\pi_v$ and $\Pi_t$ are
defined over $L$, we will also introduce the notation $\Fm_L$ for the residue field of the
ring of integers $\OC_L$ of $L$ and the corresponding notations for the sheaves.

\rem Note that the action of a geometric Frobenius element on these Harris-Taylor local systems
is given by some integral power of $q^{1/2}$: see the remark after the proposition
\ref{prop-hip} for the cohomological version of this fact.

\section{$\overline \Qm_l$-cohomology groups}
\label{para-Hi}

From now on, we fix a prime number $l$ unramified in $E$.
Let us first recall some known facts about irreducible algebraic representations of $G$,
see for example \cite{h-t} p.97. Let $\sigma_0:E \hookrightarrow
\overline{\Qm}_l$ be a fixed embedding. Let denote by $\Phi$ the set of embeddings 
$\sigma:F \hookrightarrow \overline \Qm_l$ whose restriction to $E$ equals $\sigma_0$.
There exists then an explicit bijection between irreducible algebraic representations $\xi$ of $G$ 
over $\overline \Qm_l$ and $(d+1)$-uples
$\bigl ( a_0, (\overrightarrow{a_\sigma})_{\sigma \in \Phi} \bigr )$
where $a_0 \in \Zm$ and for all $\sigma \in \Phi$, we have $\overrightarrow{a_\sigma}=
(a_{\sigma,1} \leq \cdots \leq a_{\sigma,d} )$.

For $L \subset \overline \Qm_l$ a finite extension of $\Qm_l$ such that the representation
$\iota^{-1} \circ \xi$ of highest weight
$\bigl ( a_0, (\overrightarrow{a_\sigma})_{\sigma \in \Phi} \bigr )$,
is defined over $L$, write $W_{\xi,L}$ the space of this representation and $W_{\xi,\OC_L}$
a stable lattice under the action of the maximal open compact subgroup $G(\Zm_l)$, 
where $\OC_L$ is the ring of integers of $L$ with uniformizer $\varpi_L$.

\rem if $\xi$ is supposed to be $l$-small, in the sense that for all $\sigma \in \Phi$ and all
$1 \leq i < j \leq n$ we have $0 \leq a_{\tau,j}-a_{\tau,i} < l$, then such a stable lattice is unique
up to a homothety.

\begin{nota} \label{nota-Vxi}
We will denote by $V_{\xi,\OC_L/\varpi_L^n}$ the local system on $\sh_{I,v}$ as well as
$$V_{\xi,\OC_L}=\lim_{\atopp{\longleftarrow}{n}} V_{\xi,\OC_L/\varpi_L^n} \quad \hbox{and} \quad
V_{\xi,L}=V_{\xi,\OC_L} \otimes_{\OC_L} L.$$
For the $\overline \Zm_l$ and $\overline \Qm_l$ version, we will write respectively
$V_{\xi,\overline \Zm_l}$ and $V_{\xi,\overline \Qm_l}$. 
\end{nota}
%
%

\begin{defin} 
For a $\overline \Zm_l$-sheaf $\FC$ on $\sh_{I,v}$, we will denote by 
$\FC_\xi$ the sheaf $\FC \otimes V_{\xi,\overline \Zm_l}$. We will also use similar
notations with $\OC_L$ when both $\FC$ and $\xi$ are defined other $\OC_L$.
\end{defin}

\begin{defin} \label{defi-automorphe}
Let $\xi$ be an irreducible algebraic $\Cm$-representation with finite dimension
of $G$. Then an irreducible $\Cm$-representation $\Pi_{\oo}$ of $G(\Am_{\oo})$ 
is said $\xi$-cohomological if there exists an integer $i$ such that 
$$H^i((\lie G(\Rm)) \otimes_\Rm \Cm,U_\tau,\Pi_\oo \otimes \xi^\vee) \neq (0)$$
where $U_\tau$ is a maximal open compact modulo center subgroup of $G(\Rm)$, 
cf. \cite{h-t} p.92.  We then denote by $d_\xi^i(\Pi_\oo)$ the dimension 
of this cohomology group.
\end{defin}

\rem If $\xi$ has $\bar \Qm_l$-coefficients, then an irreducible 
$\bar \Qm_{l}$-representation $\Pi^{\oo}$ of $G(\Am^{\oo})$ 
is said $\xi$-cohomolocal if there exists a $\Cm$-representation $\Pi_\oo$ 
of $G(\Am_\oo)$  such that 
$\iota_{l}\Bigl ( \Pi^{\oo} \Bigr ) \otimes \Pi_{\oo}$ is an automorphic 
$\Cm$-representation of $G(\Am)$, where $\iota_l:\overline \Qm_l \simeq \Cm$.

Recall that $G(\Qm_p) \simeq \Qm_p^\times \times
GL_d(F_v) \times \prod_{i=2}^r (B_{v_i}^{op})^\times$, for a fixed place $v$ of $F$
above $p$. For $\Pi$ an irreducible
representation of $G(\Am)$, its component for the similitude factor
$\Qm_p^\times$ is denoted as in \cite{h-t}, $\Pi_{p,0}$: as all the open compact
subgroup of $\IC$ contain $\Zm_p^\times$, the representations $\Pi$ which
appear in the cohomology groups will all verify that $(\Pi_{p,0})_{|\Zm_p^\times}=1$.
We now consider
\begin{itemize}
\item an admissible irreducible representation $\Pi$ of $G(\Am)$ with multiplicity
$m(\Pi)$ in the space of automorphic forms,

\item $\pi_v$ a cuspidal irreducible representation of $GL_g(F_v)$.

\item We also fix a finite level $I$ and we denote by $S$ the set of places $x$
such that $I_x$ is maximal and we moreover impose $v \in S$. We then denote
by $\Tm^S$ the $\overline \Zm_l$-unramified Hecke algebra outside $S$.
\end{itemize}
From \cite{boyer-compositio} and more precisely from \cite{boyer-aif}, we now recall
the main results about the $\overline \Qm_l$-cohomology groups of
Harris-Taylor perverse sheaves.
We first introduce some \textbf{notations}.
\begin{itemize}
\item For any $\overline \Qm_l$-perverse sheaf $P$ on the projective system
of schemes $\sh_{I^v(n),\bar s_v}^{=h}$ with $I^v(n):=I^vK_v(n)$, 
we consider
$$H^i(\sh_{I^v(\oo),\bar s_v},P):=\lim_{\atop{\longrightarrow}{n}} 
H^i(\sh_{I^v(n),\bar s_v},P).$$
It is equipped with an action of $\Tm^S_{\overline \Qm_l} 
\times GL_d(F_v) \times W_v$ and its image in the
Grothendieck group of admissible representations of $\Tm^S \times
GL_d(F_v) \times W_v$ will be denoted $[H^i(\sh_{I^v(\oo),\bar s_v},P)]$.

\item For a fixed maximal ideal $\widetilde{\mathfrak m}$ of 
$\Tm^S_{\overline \Qm_l}$, and a $\overline \Qm_l$-perverse
sheaf $P$ as above, we denote by  
$H^i(\sh_{I^v(\oo),\bar s_v},P)_{\widetilde{\mathfrak m}}$
the localization at $\widetilde{\mathfrak m}$
and $[H^i((\sh_{I^v(\oo),\bar s_v},P))]_{\widetilde{\mathfrak m}}$ by its image
in the Grothendieck group of admissible representations of $GL_d(F_v) \times W_v$.
If $\widetilde{\mathfrak m}$ is associated to a irrreductible representation
$\Pi^{\oo,v}$ of $G(\Am^{\oo,v})$, we might also denote it by 
$[H^i((\sh_{I^v(\oo),\bar s_v},P)]\{ \Pi^{\oo,v} \}$.

\item In particular for $HT(\pi_v,\Pi_t)$ a Harris-Taylor local system, 
we denote by 
$[H^i_c(HT(\pi_v,\Pi_t))]_{\widetilde{\mathfrak m}}$  (resp. 
$[H^i_{!*}(HT(\pi_v,\Pi_t))]_{\widetilde{\mathfrak m}}$), 
the image of $H^i(\sh_{I^v(\oo),\bar s_v},j^{=h}_! 
HT(\pi_v,\Pi_t))_{\widetilde{\mathfrak m}}$
(resp. $H^i(\sh_{I^v(\oo),\bar s_v},\lexp p j^{=h}_{!*} 
HT(\pi_v,\Pi_t))_{\widetilde{\mathfrak m}}$)
in the Grothendieck group of $GL_d(F_v) \times W_v$ admissible representations.
We also use a similar notation for $HT_{1_{tg}}(\pi_v,\Pi_t)$
by replacing $GL_d(F_v)$ by $P_{tg,d}(F_v)$. If we want to add a weight $\xi$, we then
note $[H^i_c(HT_\xi(\pi_v,\Pi_t))]_{\widetilde{\mathfrak m}}$  (resp. 
$[H^i_{!*}(HT_\xi(\pi_v,\Pi_t))]_{\widetilde{\mathfrak m}}$).
\end{itemize}

It is also well known, cf. lemma 3.2 of \cite{boyer-aif} that if $\Pi$ is an irreducible
automorphic representation of $G(\Am)$ then its local component
$\Pi_v$ is isomorphic to 
$$\speh_{s}(\st_{t_1}(\pi_{1,z}))  \times \cdots \times \speh_{s}(\st_{t_u}(\pi_{u,z})) 
\simeq \speh_s \Bigl ( \st_{t_1}(\pi_{1,z}) \times \cdots \times \st_{t_u}(\pi_{u,z}) \Bigr )$$
where the $\pi_{i,z}$ are irreducible cuspidal representations. 

\begin{nota} \label{nota-AC}
For $\pi_v$ an irreducible cuspidal unitary representation of $GL_g(F_v)$, we denote by
$\AC_{\xi,\pi_v}(r,s)$ the set of equivalence classes of automorphic irreducible 
representations of $G(\Am)$ which are $\xi$-cohomological and such that
\begin{itemize}
\item $(\Pi_{p,0})_{|\Zm_p^\times}=1$,

\item its local component at $v$ looks like $\speh_s(\st_t(\pi_v)) \times ?$
with $r=s+t-1$ and $?$ is an irreducible representation of
$GL_{d-stg}(F_v)$ which we do not want to precise.
\end{itemize}
\end{nota}

\rem An automorphic irreducible representation $\Pi$ of $G(\Am)$ 
which is $\xi$-cohomological and such that 
$$\Pi_v \simeq  \speh_s \Bigl (\st_{t_1}(\pi_{1,v}) \times \cdots \times \st_{t_u}(\pi_{u,v}) \Bigr )$$
belongs to $\AC_{\xi,\pi_{v,i}}(s+t_i-1,s)$ for $i=1,\cdots,u$.

\begin{prop} (proposition 3.6 of \cite{boyer-aif}) \label{prop-hip} \\
Let $\pi_v$ be an irreducible cuspidal representation of $GL_g(F_v)$ and
$1 \leq r \leq d/g$. Then we have
\begin{multline*}
[H^i_{!*}(HT_{\overline{1_{rg}},\xi}(\pi_v,\Pi_r))]\{ \Pi^{\oo,v} \} = 
\frac{e_{\pi_v} \sharp \ker^1(\Qm,G)}{d} \\
\sum_{\atop{(s,t)}{\Pi \in \AC_{\xi,\pi_v}(s+t-1,s)}} 
\sum_{\Pi' \in \UC_{G,\xi}(\Pi^{\oo,v})} \sum_{m_{s,t}(r,i)=1}
m(\Pi') d_\xi(\Pi'_\oo) 
\Bigl ( \Pi_r \otimes R_{\pi_v}(s,t)(r,i)(\Pi_v) \Bigr )
\end{multline*}
where
\begin{itemize}
\item $\UC_{G,\xi}(\Pi^{\oo,v})$ is the set of equivalence classes of automorphic,
$\xi$-cohomological, irreducible representations $\Pi'$ of $G(\Am)$ such that 
$(\Pi')^{\oo,v} \simeq \Pi^{\oo,v}$;

\item $\Pi_v$ is the local component of all the $\Pi' \in \UC_{G,\xi}(\Pi^{\oo,v})$
such that $d_\xi(\Pi'_\oo) \neq 0$, is the common value
of the $d_\xi^i(\Pi'_\oo)$  for $i \equiv s \mod 2$, cf. corollary VI.2.2 of \cite{h-t}
and corollary 3.3 of \cite{boyer-aif}
\end{itemize}
Concerning $R_{\pi_v}(s,t)(r,i)(\Pi_v)$ as a sum of representations of 
$GL_{d-rg}(F_v) \times \Zm$, 
for $\Pi_v \simeq  \speh_s(\st_{t_1}(\pi_{1,v})) \times \cdots \times 
\speh_s(\st_{t_u}(\pi_{u,v}))$, it is given by the formula
$$R_{\pi_v}(r,i)(\Pi_v)=\sum_{k:~\pi_{k,v} \sim_i \pi_{v}} m_{s,t_k}(r,i) 
R_{\pi_{v}}(s,t_k)(r,i)(\Pi_v,k) \otimes \Bigl ( \xi_k \otimes \Xi^{i/2} \Bigr )$$
where
\begin{itemize}
\item the characters $\xi_k$ are such that 
$\pi_{k,v} \simeq \pi_v \otimes \xi_k \circ \val \circ \det$;

\item $R_{\pi_{v}}(s,t_k)(r,i)(\Pi_v,k) $ can be written as
\begin{multline*} 
R_{\pi_{v}}(s,t_k)(r,i)(\Pi_v,k) :=\speh_s(\st_{t_1}(\pi_{1,v})) \times \cdots \times  \speh_s(\st_{t_{k-1}}(\pi_{k-1,v})) \\ 
\times R_{\pi_{k,v}}(s,t_k)(r,i) \times \\ 
\speh_s(\st_{t_{k+1}}(\pi_{k+1,v})) \times \cdots \times  \speh_s(\st_{t_u}(\pi_{u,v})).
\end{multline*}

\item $R_{\pi_{k,v}}(s,t_k)(r,i)$ is a representation of $GL_{d-rg}(F_v)$ 
which can be computed combinatorially as explained below

\item and $m_{s,t}(r,i) \in \{ 0, 1 \}$ is given in the next definition.
\end{itemize}
\end{prop}

\rem If we forget\footnote{Note that the $\xi_k \otimes \Xi^{i/2}$-eigenspaces 
can easily be identified} 
the character $\xi_k$, the action of a geometric
Frobenius element is given by multiplication by some power of $q^{1/2}$ given by the degree
of the cohomology group we are looking for.

The representation $R_{\pi_{k,v}}(s,t_k)(r,i)$ is computed as follows: we first apply
the Jacquet functor $J_{P^{op}_{rg,d}}$ to $\speh_s(\st_{t_k}(\pi_{k,v}))$ 
which can be written as a sum 
$$\sum \langle a_1 \rangle \otimes \langle a_2 \rangle$$ 
where $a_1,a_2$ are multisegments in the Zelevinsky line of $\pi_v$; 
the precise computation is given in corollary 1.5.6 of \cite{boyer-compositio}. 
For $\tau_v$ an irreducible representation of $D_{v,h}^\times$,
we then consider
$$\begin{array}{rcl}
R_{\tau_v}: \groth \Bigl ( GL_{h}(F_v) \times GL_{d-h}(F_v) \Bigr ) & \longrightarrow &
\groth \Bigl ( D_{v,h}^\times/\DC_{v,h}^\times \times GL_{d-h}(F_v) \Bigr ) \\
~ \alpha \otimes \beta &  \mapsto & \mathrm{vol} (D_{v,tg}^\times/F_v^\times)^{-1} \sum_\psi \tr \alpha
(\varphi_{\tau \otimes \psi^\vee}) \psi \otimes \beta,
\end{array}$$
where $\psi$ describe the set of character of $F_v^\times$ and 
$\varphi_{\pi_v[t]_D}$ is, as defined by Deligne, Kazhdan and Vign\'eras,
a pseudo-coefficient for $\st_t(\pi_v)$. We then write
$$R_{\pi_{k,v}[r]_D} \Bigl ( \langle a_1 \rangle \Bigr ) a_2=\sum_\psi  \epsilon_\psi \psi \otimes \Pi_\psi \in 
\groth \Bigl ( F_v^\times \times GL_{(st-r)g}(F_v) \Bigr )$$
with $\epsilon_\psi \in \{ -1,1 \}$ and where
the $\psi$ such that $\Pi_\psi$ are non zero, looks like $|-|^{j/2}$ with $j \in \Zm$.
Then we have
$$R_{\pi_{k,v}}(s,t_k)(r,i)=\Pi_{|-|^{-i/2}}.$$

\rem We do not need the precise computation of $R_{\pi_{v}}(s,t)(r,i)$, we just want to
use the fact that if $\pi'_v$ is any irreducible cuspidal representation, then 
$R_{\pi'_{v}}(s,t)(r,i)$ is obtained from $R_{\pi_{v}}(s,t)(r,i)$ by replacing $\pi_v$
by $\pi'_v$ in its combinatorial description. In particular if
the modulo $l$ reduction of $\pi'_v$ is isomorphic to the modulo $l$ reduction
of $\pi_v$, then the modulo $l$ reduction of $R_{\pi_{v}}(s,t)(r,i)$ and those of
$R_{\pi'_{v}}(s,t)(r,i)$, are isomorphic.

The assumptions on $i$ in proposition 3.6.1 of \cite{boyer-compositio} 
are contained in the following definition of $m_{s,t}(r,i)$.

\begin{defin} \label{defi-m} (cf. \cite{boyer-aif} definition 3.7) \\
The point with coordinates $(r,i)$ such that $m_{s,t}(r,i)=1$, are contained
in the convex hull of the polygon with edges 
$(s+t-1,0)$, $(t,\pm (s-1))$ and $(1, \pm (s-t))$ if $s \geq t$ (resp.
$(t-s+1,0)$ if $t \geq s$); inside it for a fixed $r$, the indexes $i$ start from
the boundary en grows by $2$, i.e. $m_{s,t}(r,i)=1$ if and only if
\begin{itemize}
\item $\max \{ 1, s+t-1-2(s-1) \}  \leq r \leq s+t-1$;

\item if $t \leq r \leq s+t-1$ then $0 \leq |i| \leq s+t-1-r$ and $i \equiv s+t-1-r \mod 2$;

\item if $\max \{ 1,s+t-1-2(s-1) \} \leq r \leq t$ then $0 \leq |i| \leq s-1-(t-r)$ and
$i \equiv s-t-1+r \mod 2$,
\end{itemize}
as represented in figures \ref{fig-coho-m1} and \ref{fig-coho-m2}.
\end{defin}

\begin{figure}[ht]
\centering
\includegraphics[scale=.8]{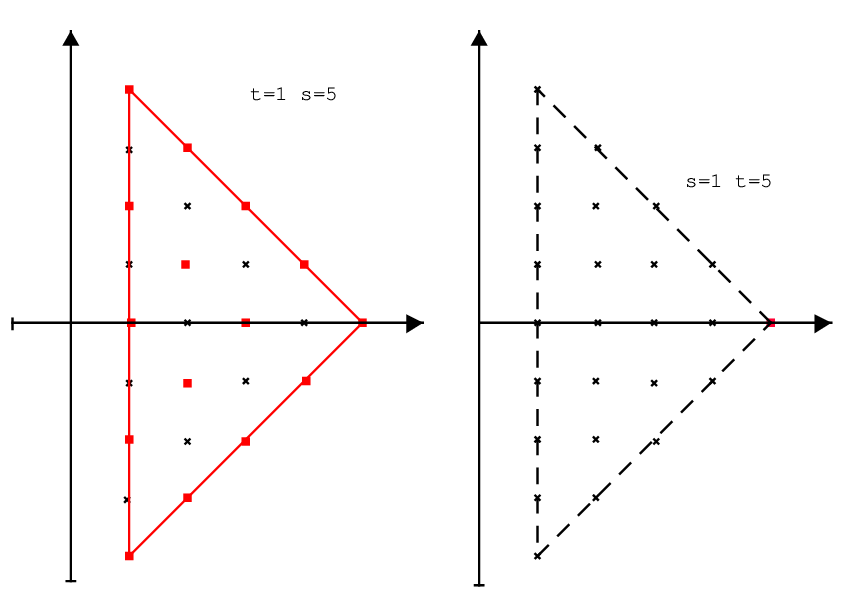}
\caption{\label{fig-coho-m1} The squares indicate the $(r,i)$ such that
$m_{s,t}(r,i)=1$ for a Speh ($t=1$) at left and a Steinberg ($s=1$) on the right}
\end{figure}

\begin{figure}[ht]
\centering
\includegraphics[scale=.8]{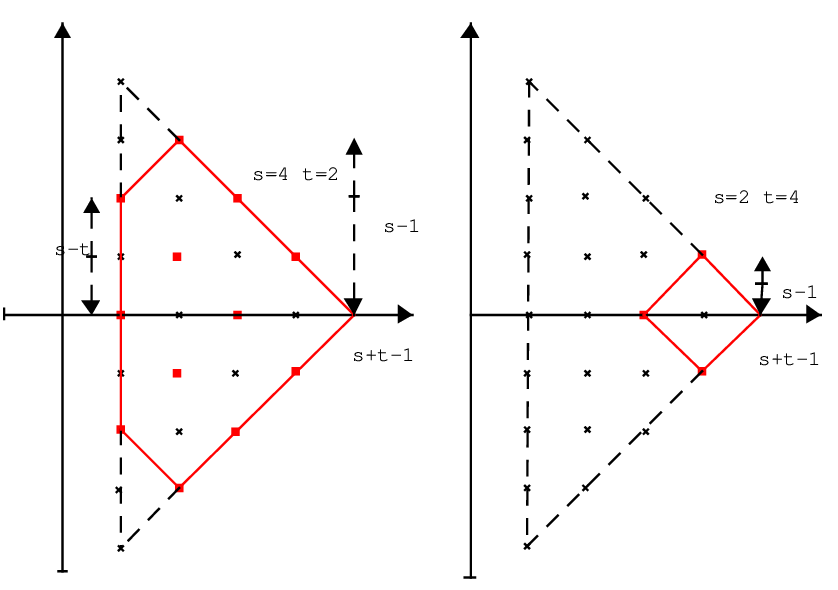}
\caption{\label{fig-coho-m2}  $m_{s,t}(r,i)=1$ when $s \geq t$ at left and 
$t \geq s$ on the right}
\end{figure}


\begin{figure}[ht]
\centering
\includegraphics[scale=.6]{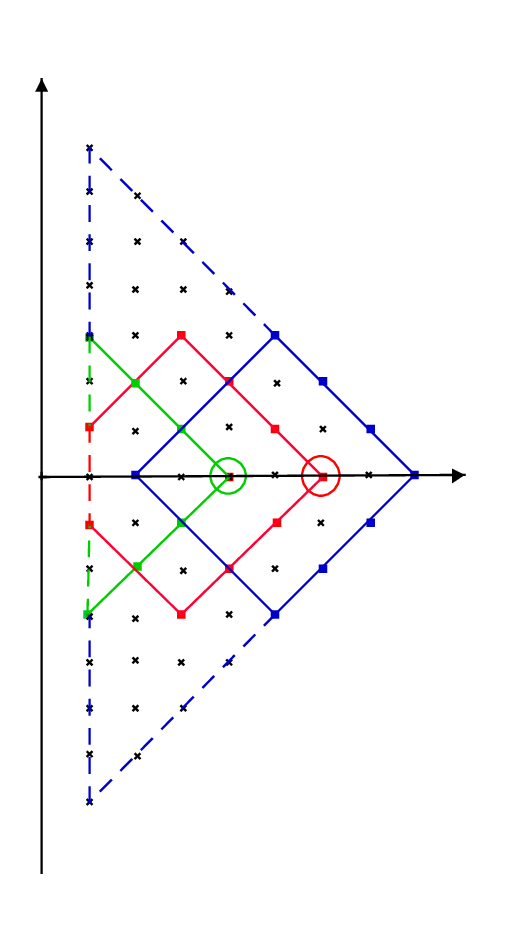}
\caption{\label{fig-coho-m4} Superposition to compute $m(r,i)$ for
$\Pi_v \simeq \speh_4(\pi_v) \times \speh_4(\st_3(\pi_v)) \times \speh_4(\st_5(\pi_v))$}
\end{figure}

\rem For
$$\Pi_v \simeq \speh_s(\st_{t_1}(\pi_{1,v})) \times \cdots \times \speh_s(\st_{t_u}(\pi_{u,v}))$$
the set of $(r,i)$ such that 
$[H^i_{!*}(HT_{\overline{1_{rg}},\xi}(\pi_v,\Pi_r))]\{ \Pi^{\oo,v} \} \neq 0$
is obtained by superposition of the $u$ previous diagrams as in the 
figure \ref{fig-coho-m4}. More precisely for a fixed $(r,i)$,
the contribution of diagram of $\speh_s(\st_{t_k}(\pi_{k,v}))$ is the same as
in the starting point $(s+t_{k}-1,0)$ after replacing $\speh_s(\st_{t_k}(\pi_{k,v}))$ by 
$R_{\pi_{k}}(s,t_{k})(r,i)$. We can then trace back any
$$R_{\pi_k}(s,t_k)(r,i)(\Pi'_v)\otimes \Bigl ( \xi_k \otimes \Xi^{i/2} \Bigr )$$ 
of $[H^i_{!*}(HT_{\overline{1_{rg}},\xi}(\pi_v,\Pi_r))]\{ \Pi^{\oo,v} \}$, 
to 
$$R_{\pi_k}(s+t_k,0)(r',0)(\Pi'_v) \otimes \Bigl ( \xi_k \otimes \Xi^{0} \Bigr )$$ 
of $[H^0(\lexp p j^{\geq (s+t_k)g}_{!*} \FC_{\bar \Qm_l,\xi}(\pi_v,s+t_k)_1[d-(s+t_k)g])]\{ \Pi^{\oo,v} \}$.
Note although that for $i=0$, some of the constituants of 
$[H^0_{!*}(HT_{\overline{1_{rg}},\xi}(\pi_v,\Pi_r))]\{ \Pi^{\oo,v} \}$ 
may or may not come from $r' >r$.

\noindent \textit{Comments about the example of figure \ref{fig-coho-m4} for $r=4$}: 
\begin{itemize}
\item $ \speh_4(\pi_v) \times \speh_4(\st_3(\pi_v)) \times R_{\pi_v}(4,5)(4,0)$ 
comes from $(8,0)$;

\item $ \speh_4(\pi_v) \times  R_{\pi_v}(4,3)(4,0) \times \speh_4(\st_5(\pi_v))$ 
comes from $(6,0)$;

\item $R_{\pi_v}(4,1)(4,0) \times \speh_4(\st_3(\pi_v)) \times \speh_4(\st_5(\pi_v))$ 
does not come from any $(r',0)$ for $r'>4$.
\end{itemize}

Concerning $[H^i_{c}(HT_{\overline{1_{rg}},\xi}(\pi_v,\Pi_r))]\{ \Pi^{\oo,v} \}$ we have
a exactly the same computation replacing the intergers $m_{s,t}(r,i)$ by
$n_{s,t}(r,i)$ defined as follows, cf. \cite{boyer-aif} definition 3.11.
The point with coordinates $(r,i)$ such that $n_{s,t}(r,i)=1$ are those in the convex
hull of the polygon with edges $(s+t-1,0)$, $(s,0)$, $(1,s-1)$ and $(t,s-1)$.

\begin{figure}[ht]
\centering
\includegraphics{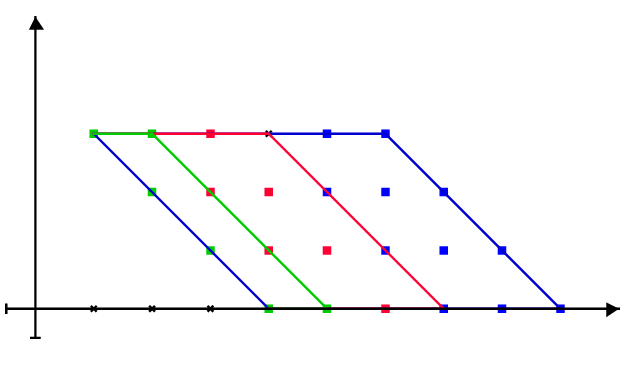}
\caption{\label{fig-coho-n4} Superposition to compute $n(r,i)$ for
$\Pi_v \simeq  \speh_4(\pi_v) \times \speh_4(\st_3(\pi_v)) \times \speh_4(\st_5(\pi_v))$}
\end{figure}

\begin{prop} (proposition 3.12 of \cite{boyer-aif}) \label{prop-hic} \\
Let $\pi_v$ be an irreducible cuspidal representation of $GL_g(F_v)$ and
$1 \leq r \leq d/g$. Then we have
\begin{multline*}
[H^i_{!c}(HT_{\overline{1_{rg}},\xi}(\pi_v,\Pi_r))]\{ \Pi^{\oo,v} \} = 
\frac{e_{\pi_v} \sharp \ker^1(\Qm,G)}{d} \\
\sum_{\atop{(s,t)}{\Pi \in \AC_{\xi,\pi_v}(s+t-1,s)}} 
\sum_{\Pi' \in \UC_{G,\xi}(\Pi^{\oo,v})} \sum_{m_{s,t}(r,i)=1}
m(\Pi') d_\xi(\Pi'_\oo) 
\Bigl ( \Pi_r \otimes S_{\pi_v}(s,t)(r,i)(\Pi_v) \Bigr )
\end{multline*}
where $S_{\pi_v}(s,t)(r,i)(\Pi_v)$ is a sum of representations of 
$GL_{d-rg}(F_v) \times \Zm$, 
for $\Pi_v \simeq  \speh_s(\st_{t_1}(\pi_{1,v})) \times \cdots \times 
\speh_s(\st_{t_u}(\pi_{u,v}))$, it is given by the formula
$$S_{\pi_v}(r,i)(\Pi_v)=\sum_{k:~\pi_{k,v} \sim_i \pi_{v}} n_{s,t_k}(r,i) 
R_{\pi_{v}}(s,t_k)(r,i)(\Pi_v,k) \otimes \Bigl ( \xi_k \otimes \Xi^{i/2} \Bigr )$$
with the notations of \ref{prop-hip} and $n_{s,t}(r,i) \in \{ 0, 1 \}$ is given above.
\end{prop}

We fix a level $I^v$ outside $v$ and a maximal ideal $\mathfrak m$
of $\Tm^S$: recall that fixing $\mathfrak m$ means that we focus on liftings 
$\widetilde{\mathfrak m} \subset \mathfrak m$, i.e. on $\xi$-cohomological
automorphic representations $\Pi$ of $G(\Am)$ such that their modulo $l$
Satake parameters are prescribed by $\mathfrak m$. As usual
we then denote with an
index $\mathfrak m$ for the localization $\otimes_{\Tm^S} \Tm^S_{\mathfrak m}$.

\section{Integral Harris-Taylor perverse sheaves}

Let $\pi_v$ be an irreducible supercuspidal representation $\pi_v$ defined over some finite
extension $L/\Qm_l$.
When its modulo $\varpi_L$ reduction is still supercuspidal, then for any
$t \geq 1$, the representation $\pi_v[t]_D$, also defined over $L$, 
remains irreducible modulo $\varpi_L$ so that
$\LC(\pi_v[t]_D)_{\overline{1_{tg}}}$ has an unique, up to homothety, stable
$\OC_L$-lattice. For any $\OC_L$-representation $\Pi_t$ 
of $GL_{tg}(F_v)$,
we have then a well defined $\OC_L$-local system $HT_{\OC_L}(\pi_v,\Pi_t)$.

\rem $\Pi_t$ is called the infinitesimal part of the Harris-Taylor local system and
it does not play any role here, we only mention it as it appears everywhere
in \cite{boyer-aif} or \cite{boyer-invent2}.

Over $\OC_L$, we have two natural $t$-structures denoted $p$ and $p+$
which are exchanged by the Grothendieck-Verdier duality. 
We then have two notions of intermediate extension,
$\lexp p j_{!*}$ and $\lexp {p+} j_{!*}$. In \cite{boyer-torsion} we explain,
using the Newton stratification, how to construct $\OC_L$-filtrations
of perverse sheaves, with torsion free graded parts. In \cite{boyer-duke}, read the introduction of loc. cit. for the statement of the results, we
prove the following result.

\begin{thm}
When the modulo $\varpi_L$ 
reduction of $\pi_v$ remains supercuspidal, the two intermediate
extensions $\lexp p j^{=tg}_{!*} HT_{\OC_L}(\pi_v,\Pi_t)$ and 
$\lexp {p+} j^{=tg}_{!*} HT_{\OC_L}(\pi_v,\Pi_t)$ are equals.
\end{thm}

\rem When $\pi_v$ is a character this result is essentially formal otherwise it is rather difficult to prove it.
When the modulo $\varpi_L$ reduction of $\pi_v$ is no more supercuspidal, in loc. cit.
we describe the $\varpi_L$-torsion of the cokernel of
$\lexp p j^{=tg}_{!*} HT_{\OC_L}(\pi_v,\Pi_t) \hookrightarrow \lexp {p+} j^{=tg}_{!*} 
HT_{\OC_L}(\pi_v,\Pi_t)$.

The following resolution of $\lexp p j_{!*}^{=tg} HT(\pi_{v},\Pi_t)$, proved
over $\overline \Qm_l$ in \cite{boyer-invent2}, is still valid over $\OC_L$, cf. \cite{boyer-duke}:
\addtocounter{thm}{1}
\begin{multline} \label{eq-resolution0}
0 \rightarrow j_!^{=sg} HT_{\OC_L}(\pi_{v},\Pi_t \{ \frac{t-s}{2} \}  \times 
\speh_{s-t}(\pi_{v}\{ t/2 \} ))
 \otimes \Xi^{\frac{s-t}{2}} \longrightarrow \cdots  \\
\longrightarrow j_!^{=(t+1)g} HT_{\OC_L}(\pi_{v},\Pi_t \{ -1/2 \}  \times \pi_{v} \{ t/2 \} ) 
\otimes \Xi^{\frac{1}{2}} \longrightarrow \\ j_!^{=tg} HT_{\OC_L}(\pi_{v},\Pi_t) 
\longrightarrow  \lexp p j_{!*}^{=tg} HT_{\OC_L}(\pi_{v},\Pi_t) \rightarrow 0.
\end{multline}
This resolution is equivalent to the fact that the cohomology sheaves of 
$\lexp p j_{!*}^{=tg} HT_{\OC_L}(\pi_{v},\Pi_t)$ are torsion free. 

Note that when the modulo $\varpi_L$ reduction of $\pi_v$ remains supercuspidal, the
lattices of the Harris-Taylor local systems are uniquely well defined. Otherwise,
in the previous formula any stable lattice of $HT_{\OC_L}(\pi_v,\Pi_t)$ then define
a uniquely well defined lattice of the others Harris-Taylor local systems of the formula:
we do not need to precise them in the following and we prefer not to introduce symbols
in the notation to explicit this fact.

We have a filtration 
\addtocounter{thm}{1}
\begin{multline} \label{eq-fil!}
\Fil^{(t-s)}_!(\pi_v,\Pi_t) \hookrightarrow \Fil^{t-s+1}_!(\pi_v,\Pi_t) \hookrightarrow 
\cdots  \\
\hookrightarrow \Fil^0_!(\pi_v,\Pi_t)=j^{= tg}_{!} HT_{\OC_L}(\pi_v,\Pi_t),
\end{multline}
with graded parts 
$$\gr^{-\delta}_!(\pi_v,\Pi_t) \simeq 
\lexp p j^{\geq (t+\delta)g}_{!*} HT_{\OC_L}(\pi_v,\Pi_t \overrightarrow{\times}
\st_\delta(\pi_v)) \otimes \Xi^{\delta/2}.$$

\section{Torsion in the cohomology of a HT perverse sheaf}

\begin{lemma} 
The action of a geometric Frobenius element on
$$H^{i}_{tor}(\sh_{I,\bar s_v},\lexp p j^{=tg}_{!*} HT_{\OC_L}(\pi_v,\Pi_t)) \otimes_{\OC_L} \Fm_L$$
or on
$$H^{i}_{tor}(\sh_{I,\bar s_v},\lexp p j^{=tg}_{!} HT_{\OC_L}(\pi_v,\Pi_t)) \otimes_{\OC_L} \Fm_L$$
decomposes these spaces into isotypic subspaces where it acts through
$\xi_k(1) q^{r/2}$ for various integers $r \in \Zm$ and where the $\xi_k$ are as in proposition
\ref{prop-hip}.
%
\end{lemma}

Precisely we mean that the torsion submodules decompose into a direct sum of 
$\OC_L/\varpi_L^a \OC_L$
for various $a \geq 1$ where the geometric Frobenius element acts through some $\xi_k(1) q^{r/2}$
where $\xi_k$ modulo $l^a$ can be recovered by considering the action of $GL_d(F_v)$.

\begin{proof}
We argue by induction on $t$ from $s_g$ to $1$. For $t=s_g$ the cohomology groups are torsion
free so the result follows from the description of the free quotient given by proposition \ref{prop-hip}.

Suppose now that the lemma is true for every $t > t_0$. We then look at the 
$\OC_L$-cohomology of $\lexp p j^{=t_0g}_{!*} HT_{\OC_L}(\pi_v,\Pi_{t_0}))$ through 
the spectral sequence associated to the resolution (\ref{eq-resolution0}).
Note that all the $E_1^{p,q}$ terms verify the property of the lemma except maybe for the
$H^i(\sh_{I,\bar s_v},\lexp p j^{=t_0g}_{!} HT_{\OC_L}(\pi_v,\Pi_{t_0}))$ when $i>0$: indeed these
cohomology groups are zero for $i<0$ and torsion free for $i=0$. We then deduce that all the
$E_1^{p,q}$ satisfy the property when $p+q \leq 0$ so that it is the same for
$H^i(\sh_{I,\bar s_v},\lexp p j^{=t_0g}_{!} HT_{\OC_L}(\pi_v,\Pi_{t_0}))$ when $i \leq 0$ and by duality,
as $\lexp p j^{=t_0g}_{!} HT_{\OC_L}(\pi_v,\Pi_{t_0}) \simeq 
\lexp {p+} j^{=t_0g}_{!} HT_{\OC_L}(\pi_v,\Pi_{t_0})$,
it is true for all $i$.

We then look at the spectral sequence associated to the resolution (\ref{eq-fil!}) where now all the
$E_1^{p,q}$ terms verify the property so that it is the same for the 
$H^i(\sh_{I,\bar s_v},\lexp p j^{=t_0g}_{!} HT_{\OC_L}(\pi_v,\Pi_{t_0}))$.

\end{proof}

Consider now two cuspidal representations $\pi_v$ and $\pi'_v$, defined over $\OC_L$
for some finite extension $L/\Qm_l$, with isomorphic modulo 
$\varpi_L$-reduction of $GL_g(F_v)$ with $s_g:=\lfloor \frac{d}{g} \rfloor \geq 1$. 
For every $1 \leq t \leq s_g$ the two Harris-Taylor local systems $HT_{\Fm_L}(\pi_v,t)$ and 
$HT_{\Fm_L}(\pi'_v,t)$
are then isomorphic and from the main result of \cite{boyer-duke} we have
$$\lexp p j^{=tg}_{!*} HT_{\OC_L}(\pi_v,t) \simeq \lexp {p+} j^{=tg}_{!*} HT_{\OC_L}(\pi_v,t)$$
and
$$\lexp p j^{=tg}_{!*} HT_{\OC_L}(\pi'_v,t) \simeq \lexp {p+} j^{=tg}_{!*} HT_{\OC_L}(\pi'_v,t),$$
which implies
$$\Fm \lexp p j^{=tg}_{!*} HT_{\OC_L}(\pi_v,t) \simeq \lexp {p} j^{=tg}_{!*} \Fm HT_{\OC_L}(\pi_v,t)
\simeq \Fm \lexp p j^{=tg}_{!*} HT_{\OC_L}(\pi_v,t) .$$
We then have the following diagram
$$\xymatrix{
\Fm H^{i}(\sh_{I^v,\bar s_v},\lexp p j^{=tg}_{!*} HT_{\OC_L}(\pi_v,\Pi_t)) \ar@{^{(}->}[d] &
\Fm H^{i}(\sh_{I^v,\bar s_v},\lexp p j^{=tg}_{!*} HT_{\OC_L}(\pi'_v,\Pi_t)) \ar@{^{(}->}[d] 
\\
H^{i}(\sh_{I,\bar s_v},\lexp p j^{=tg}_{!*} \Fm HT_{\OC_L}(\pi_v,\Pi_t)) \ar@{->>}[d]  \ar[r]_{\sim}^{\iota_i}&
H^{i}(\sh_{I,\bar s_v},\lexp p j^{=tg}_{!*} \Fm HT_{\OC_L}(\pi'_v,\Pi_t)) \ar@{->>}[d] &
\\ H^{i+1}(\sh_{I,\bar s_v},\lexp p j^{=tg}_{!*} HT_{\OC_L}(\pi_v,\Pi_t))[\varpi_L] & 
H^{i+1}(\sh_{I,\bar s_v},\lexp p j^{=tg}_{!*} HT_{\OC_L}(\pi'_v,\Pi_t))[\varpi_L]
}$$

\begin{lemma} \label{lem-torsion-min}
For any $i$, 
in the Grothendieck group of $\overline \Fm_l[GL_d(F_v)]$-admissible modules, the image of
$$\lim_{\atop{\rightarrow}{n}} H^{i+1}(\sh_{I^v(n),\bar s_v},\lexp p j^{=tg}_{!*} HT_{\OC_{L_n}}(\pi_v,\Pi_t))_{\mathfrak m}[\varpi_{L_n}]
\otimes_{\Fm_L} \overline \Fm_l$$
is equal to those of  
$$\lim_{\atop{\rightarrow}{n}} H^{i+1}(\sh_{I^v(n),\bar s_v},\lexp p j^{=tg}_{!*} HT_{\OC_{L_n}}(\pi'_v,\Pi_t))_{\mathfrak m}[\varpi_{L_n}]
\otimes_{\Fm_L} \overline \Fm_l,$$
where $L_n/\Qm_l$ is a finite extension big enough so that everything is well defined over it.
\end{lemma}

\begin{proof}
We first argue at finite level $I^v(n)$ using the above isomorphisms $\iota_i$.
Recall that each cohomology group is the extension of its free quotient by its
torsion sub-module:
\begin{multline*}
H^{i}_{tor}(\sh_{I^v(n),\bar s_v},\lexp p j^{=tg}_{!*} HT_{\OC_L}(\pi_v,\Pi_t)) \hookrightarrow
H^{i}(\sh_{I^v(n),\bar s_v},\lexp p j^{=tg}_{!*} HT_{\OC_L}(\pi_v,\Pi_t)) \\ 
\twoheadrightarrow
H^{i}_{free}(\sh_{I^v(n),\bar s_v},\lexp p j^{=tg}_{!*} HT_{\OC_L}(\pi_v,\Pi_t)).
\end{multline*}
The action of a geometric Frobenius $\frob_v^{-1}$ 
on the free quotient is well understood as it is pure and so given by multiplication by
$\xi_k(1)q^{-i/2}$. To identify
the torsion submodules through $\iota_i$ we then focus on the action of $\frob_v$ on it. To do it
we first consider $i_0$ such that for every $i > i_0$ the torsion submodules of 
$H^{i}(\sh_{I^v(n),\bar s_v},\lexp p j^{=tg}_{!*} HT_{\OC_L}(\pi_v,\Pi_t))$ and
$H^{i}(\sh_{I^v(n),\bar s_v},\lexp p j^{=tg}_{!*} HT_{\OC_L}(\pi'_v,\Pi_t))$ are zero whatever is 
$n$.\footnote{Note that, cf. \cite{aaa}, $i_0 = s-t$ always work for example.}
Then for every $r \neq i_0$, note that $\iota_{i_0}$ identifies the parts of 
$$\Fm H^{i_0}(\sh_{I^v(n),\bar s_v},\lexp p j^{=tg}_{!*} HT_{\OC_L}(\pi_v,\Pi_t))$$ 
where $\frob_v$
acts through $\xi_k(1) q^{r/2}$ with the same part of
$$\Fm H^{i_0}(\sh_{I^v(n),\bar s_v},\lexp p j^{=tg}_{!*} HT_{\OC_L}(\pi'_v,\Pi_t)).$$ 
By purity these parts come from
the torsion submodule of the integral cohomology group.

We then focus on the part where $\frob_v$ acts by multiplication by $\xi_k(1) q^{i_0/2}$ and we want to
prove that for every $i \leq i_0$, the image of map the associated part of 
$$\lim_{\atop{\rightarrow}{n}} H^{i}(\sh_{I^v(n),\bar s_v},\lexp p j^{=tg}_{!*} 
HT_{\OC_{L_n}}(\pi_v,\Pi_t))_{\mathfrak m}[\varpi_{L_n}] \otimes_{\Fm_L} \overline \Fm_l$$ 
in the Grothendieck group of $\overline \Fm_l[GL_d(F_v)]$-admissible modules, is equal to those of
$$\lim_{\atop{\rightarrow}{n}} H^{i}(\sh_{I^v(n),\bar s_v},\lexp p j^{=tg}_{!*} HT_{\OC_{L_n}}(\pi'_v,\Pi_t))[\varpi_{L_n}] \otimes_{\Fm_L} \overline \Fm_l.$$ 
To do so we argue by induction from $i=-i_0$ to $i_0$.

- For $i=-i_0$, by Grothendieck-Verdier duality, 
$$H^{i}(\sh_{I^v(n),\bar s_v},\lexp p j^{=tg}_{!*} HT_{\OC_{L_n}}(\pi_v,\Pi_t)) \hbox{ and } 
H^{i}(\sh_{I^v(n),\bar s_v},\lexp p j^{=tg}_{!*} HT_{\OC_{L_n}}(\pi'_v,\Pi_t))$$ 
are torsion free. 
As the free quotient does not contribute to the $\xi_k(1) q^{i_0/2}$-eigenspaces of $\frob_v$
in $H^{i}(\sh_{I^v(n),\bar s_v},\lexp p j^{=tg}_{!*} \Fm HT_{\OC_{L_n}}(\pi_v,\Pi_t))$, we then deduce that
$\iota_{-i_0}$ identifies the $\xi_k(1) q^{i_0/2}$-eigenspaces of $\frob_v$ inside
$H^{i+1}(\sh_{I^v(n),\bar s_v},\lexp p j^{=tg}_{!*} HT_{\OC_{L_n}}(\pi_v,\Pi_t))[\varpi_{L_n}]$ with those of
$H^{i+1}(\sh_{I^v(n),\bar s_v},\lexp p j^{=tg}_{!*} HT_{\OC_{L_n}}(\pi'_v,\Pi_t))[\varpi_{L_n}]$.

- Suppose the statement true for $i$ and consider $i+1\leq i_0-1$. Note first that
the free quotients do not contribute. The torsion we are interested in, is a direct sum 
of $\OC_{L_n}/\varpi_L^k \OC_{L_n}$
where the action of $\frob_v$ is given by $\xi_k(1) q^{i_0/2}$
so that the associated $\varpi_{L_n}$-torsion is 
$\varpi_{L_n}^{k-1} \OC_{L_n}/\varpi_{L_n}^k \OC_{L_n}$ and its
image by $\Fm$ is $\Fm_{L_n}$. The
$\xi_k(1) q^{i_0/2}$-eigenspaces of $\frob_v$ inside the torsion of
$$H^{i+1}(\sh_{I^v(n),\bar s_v},\lexp p j^{=tg}_{!*} HT_{\OC_{L_n}}(\pi_v,\Pi_t))_{\mathfrak m},$$ 
$$\hbox{resp. } 
H^{i+1}(\sh_{I^v(n),\bar s_v},\lexp p j^{=tg}_{!*} HT_{\OC_{L_n}}(\pi'_v,\Pi_t))_{\mathfrak m}$$ 
can be written
$$\bigoplus_{i=1}^{a_n} \OC_{L_n}/\varpi_{L_n}^{k_i} \OC_{L_n}, \quad \hbox{resp. }
\bigoplus_{i=1}^{a'_n} \OC_{L_n}/\varpi_{L_n}^{k'_i} \OC_{L_n}$$
where, by the induction hypothesis, $a_n=a'_n$ and such that, after taking the
$\varpi_{L_n}$-torsion, extending the scalar to $\overline \Fm_l$ and the inductive limit over $n$,
the associated image in the Grothendieck group
of $\overline \Fm_l[GL_d(F_v)]$-admissible modules, are equal.

We then deduce, as the free quotients do not interfere and using the equivariance of the
$\iota_i$, the $\xi_k(1) q^{i_0/2}$-eigenspaces of $\frob_v$ inside
$\Fm H^{i+1}(\sh_{I^v(n),\bar s_v},\lexp p j^{=tg}_{!*} HT_{\OC_{L_n}}(\pi_v,\Pi_t))_{\mathfrak m}$ 
organize themselves when $n$ varies, such that after extending the scalars to $\overline \Zm_l$,
its image in the Grothendieck group of $\overline \Fm_l[GL_d(F_v)]$-admissible modules is
equal to those given by the
$\Fm H^{i+1}(\sh_{I^v(n),\bar s_v},\lexp p j^{=tg}_{!*} HT_{\OC_{L_n}}(\pi'_v,\Pi_t))_{\mathfrak m}$.
Using the vertical short exact sequences of the previous diagram, we then deduce that
the $\xi_k(1) q^{i_0/2}$-eigenspaces of $\frob_v$ of 
$H^{i+2}(\sh_{I^v(n),\bar s_v},\lexp p j^{=tg}_{!*} HT_{\OC_{L_n}}(\pi_v,\Pi_t))[\varpi_{L_n}]$ 
after extending the scalar to $\overline \Fm_l$ and the inductive limit over $n$, has an image in
 the Grothendieck group of $\overline \Fm_l[GL_d(F_v)]$-admissible modules which is equal to those
given by
$H^{i+2}(\sh_{I^v(n),\bar s_v},\lexp p j^{=tg}_{!*} HT_{\OC_{L_n}}(\pi'_v,\Pi_t))[\varpi_{L_n}]$.

- In particular for $i=i_0-1$ we know that the $\xi_k(1) q^{i_0/2}$-eigenspace of $\frob_v$ of 
$H^{i_0}(\sh_{I^v(n),\bar s_v},\lexp p j^{=tg}_{!*} HT_{\OC_{L_n}}(\pi_v,\Pi_t))_{\mathfrak m}
[\varpi_{L_n}]$, after tensoring to $\overline \Fm_l$ and taking the inductive limit over $n$,
has an image in the Grothendieck group of $\overline \Fm_l[GL_d(F_v)]$-admissible modules,
which is equal to those given by the
$H^{i_0}(\sh_{I^v(n),\bar s_v},\lexp p j^{=tg}_{!*} HT_{\OC_{L_n}}(\pi'_v,\Pi_t))_{\mathfrak m}
[\varpi_{L_n}]$.

Finally we can forget about the torsion for the cohomology groups of index $i_0$ and argue by induction
to obtain the result.
\end{proof}

\section{Automorphic congruences}
\label{para-congruence}

Consider 
\begin{itemize}
\item two irreducible algebraic representations $\xi$ and $\xi'$ defined over
some finite extension $L/\Qm_l$ with stable lattice $W_{\xi,\OC_L}$
and $W_{\xi',\OC_L}$ such that they become isomorphic modulo $\lambda$
any uniformizer of $\OC_L$;

\item two irreducible cuspidal representations 
$\pi_v$ and $\pi'_v$, defined over $\OC_L$, such that their modulo $\varpi_L$ reduction are isomorphic
and supercuspidal,

\item and a maximal ideal $\mathfrak m$ of $\Tm^S$.
\end{itemize}
For $V$ a $\OC_L$-free module, we denote by 
$$r_l(V)=V \otimes_{\OC_L} \Fm_L \otimes_{\Fm_L} \overline \Fm_l$$ 
its modulo $l$ reduction. When $V$ is a $\overline \Zm_l$-free module, we denote
$r_l(V)$ for the associated $\overline \Fm_l$-vector space.

\begin{nota}
For $\psi_v$ a $\Fm_L$-representation of $GL_h(F_v)$,
we denote by $\dim_{n} \psi$ the set
$$\Bigl \{ \dim_{\Fm_L} \psi_v^{K_v(n)},~n \in \Nm \Bigr \}.$$
\end{nota}

\begin{thm} \label{thm-main} (cf. conjecture 5.2.1 of \cite{boyer-aif}) \\
For every $r \geq 1$, we have the following equality
\addtocounter{thm}{1}
\begin{multline}
\label{eq-thm}
\sum_{\Pi \in \AC_{\xi,\pi_v}(r,s)} m(\Pi) d_\xi(\Pi_\oo)  
\dim_{\overline \Qm_l} (\Pi^{\oo,v})^{I^v} \dim_{n} r_l\bigl ( 
R_{\pi_v}(r,r)(\Pi_v) \bigr ) \\ = \\
\sum_{\Pi' \in \AC_{\xi',\pi'_v}(r,s)} m(\Pi') d_\xi(\Pi'_\oo)  
\dim_{\overline \Qm_l} (\Pi^{',\oo,v})^{I^v} \dim_{n} r_l\bigl ( 
R_{\pi_v}(r,r)(\Pi'_v) \bigr ).
\end{multline}
\end{thm}

A qualitative version of the previous theorem could be formulated as follows.
\begin{itemize}
\item Given an irreducible $\xi$-cohomological
automorphic representation $\Pi$ of $G(\Am)$ such that
$\Pi_v \simeq \speh_s(\st_t(\pi_v)) \times ?$, with $\pi_v$ cuspidal
with modulo $l$ reduction still supercuspidal,

\item and $\pi'_v$ such that its modulo $l$ reduction is isomorphic to those of
$\pi_v$,
\end{itemize}
then there exists an irreducible $\xi'$-cohomological automorphic representation
$\Pi'$ such that
\begin{itemize}
\item outside $v$, $\Pi$ and $\Pi'$ share the same level $I^v$,

\item their also share the same modulo $\varpi_L$ Satake parameters\footnote{where $L/\Qm_l$
is a finite extension such that $\xi$, $\pi_v$ and $\pi'_v$ are defined for some fixed finite level} 
at the places outside
$S$, i.e. they are weakly automorphic congruent,

\item at $v$, we have $\Pi'_v \simeq \speh_s(\st_t(\pi_v)) \times ?'$,
\end{itemize}
where by convention, the symbols $?$ and $?'$ mean 
any representation we do not want to precise.
%
%

\begin{proof}
We first consider \textit{the case where $r=r_{\max}$ is maximal} such that there exists $s$ and 
$\widetilde{\mathfrak m} \subset \mathfrak m$ with 
$\Pi_{\widetilde{\mathfrak m}} \in \AC_{\xi,\pi_v}(r,s)$
and $\Pi_{\widetilde{\mathfrak m},S}^{I^v} \neq (0)$. 
We then look at the free quotient of
$H^0(\sh_{I^v(n),\bar s_v},\lexp p j^{=rg}_{!*}HT_{\OC_{L_n},\xi}(\pi_v,\Pi_t))_{\mathfrak m}$ which 
can be described, after tensoring with $\overline \Qm_l$, as the sum of the contributions of the irreducible
automorphic representations of $\AC_{\xi,\pi_v}(r,s)$ with $r=s+t-1$ for
some finite set 
\begin{multline*}
\BC_\xi(\pi_v,r):=
\Bigl \{ (s,t) \hbox{ s. t. } s+t-1=r, \\ 
\exists \widetilde{\mathfrak m} \subset \mathfrak m,
\Pi_{\widetilde{\mathfrak m}} \in \AC_{\xi,\pi_v}(r,s) \hbox{ and }
(\Pi_{\widetilde{\mathfrak m}}^{\oo,v})^{I^v} \neq (0) \Bigr \}.
\end{multline*}
The qualitative version of the theorem asks to prove that 
$\BC_\xi(\pi_v,r)=\BC_{\xi'}(\pi'_v,r)$ for every $r$,
and the quantitative one then follows from the formula of the multiplicities
in \ref{prop-hip}.

(a) Consider first the elements $(s,t) \in \BC_{\xi}(\pi_v,r_{\max})$.
 The image, in the Grothendieck group of $\overline \Fm_l[GL_d(F_v)]$-admissible modules, of
\begin{multline*}
[\lim_{\atop{\rightarrow}{n}} H^0(\sh_{I^v(n),\bar s_v},\lexp p j^{=rg}_{!*}HT_{\OC_{L_n},\xi}(\pi_v,\Pi_t))_{\mathfrak m} \otimes_{\OC_{L_n}} \Fm_{L_n} \otimes_{\Fm_{L_n}} \overline \Fm_l] \\
=[\lim_{\atop{\rightarrow}{n}}H^*(\sh_{I^v(n),\bar s_v},\lexp p j^{=rg}_{!*}HT_{\OC_{L_n},\xi}(\pi_v,\Pi_t))_{\mathfrak m}\otimes_{\OC_{L_n}} \Fm_{L_n} \otimes_{\Fm_{L_n}} \overline \Fm_l]
\end{multline*}
is equal to that given by the
$H^*(\sh_{I^v(n),\bar s_v},\lexp p j^{=rg}_{!*}HT_{\OC_{L_n},\xi'}(\pi'_v,\Pi_t))_{\mathfrak m}$
which is also equal to that given by the
$H^0(\sh_{I^v(n),\bar s_v},\lexp p j^{=rg}_{!*}HT_{\OC_{L_n},\xi'}(\pi'_v,\Pi_t))_{\mathfrak m}$.
We can then conclude, as their torsion parts are equal in the Grothendieck group
of $\overline \Fm_l[GL_d(F_v)]$-admissible modules\footnote{Note in this case,
that the two cohomology groups are torsion free},
from the explicit computation of the multiplicities in \ref{prop-hip}, that the contribution of the
elements of $\BC_{\xi}(\pi_v,r_{\max})$ is the same as those of $\BC_{\xi'}(\pi'_v,r_{\max})$, 
which
corresponds to the quantitative version of our statements for all elements $(s,t)$ such that
$s+t-1=r_{\max}$.

(b) To separate the contributions of the $(s,t) \in \BC_{\xi}(\pi_v,r_{\max})$,
we now focus on the sequence of the dimensions 
$$d_{k,n}(\pi_v)=\dim_{\overline \Fm_l} H^k(\sh_{I^v(n),\bar s_v},\lexp p j^{=(r-k)g}_{!*}
HT_{\OC_{L_n},\xi}(\pi_v,\mathds 1_{r-k}))_{\mathfrak m} \otimes_{\OC_{L_n}} \overline \Fm_l,$$
where $\mathds 1_{r-k}$ is the trivial representation of $GL_{(r-k)g}(F_v)$, for
$k=0,\cdots,r-1$ and $n \in \Nm$. More specifically we will focus on the 
$d_{k,n}(\pi_v)-d_{k+1,n}(\pi_v)$. Note the following facts:
\begin{itemize}
\item By maximality of $r=r_{\max}$, we have
$$H^{k+1}(\sh_{I^v(n),\bar s_v},\lexp p j^{=(r-k)g}_{!*}
HT_{\OC_{L_n},\xi}(\pi_v,\mathds 1_{r-k}))_{\mathfrak m} =(0)$$ 
so that
\begin{multline*}
H^k(\sh_{I^v(n),\bar s_v},\lexp p j^{=(r-k)g}_{!*}
HT_{\OC_{L_n},\xi}(\pi_v,\mathds 1_{r-k}))_{\mathfrak m} \otimes_{\OC_{L_n}} \overline \Fm_l
\\ \simeq H^k(\sh_{I^v(n),\bar s_v},j^{=(r-k)g}_{!*}
HT_{\overline \Fm_l,\xi} (\pi_v,\mathds 1_{r-k}))_{\mathfrak m}.
\end{multline*}

\item  By lemma \ref{lem-torsion-min}, the modulo $\varpi_{L_n}$ reduction
of the torsion of 
$$H^k(\sh_{I^v(n),\bar s_v},\lexp p j^{=(r-k)g}_{!*}
HT_{\OC_{L_n},\xi}(\pi_v,\mathds 1_{r-k}))_{\mathfrak m},$$
after extending the scalar to $\overline \Fm_l$ and taking the inductive limit over $n$, has its
image in the Grothendieck group of $\overline \Fm_l[GL_d(F_v)]$-admissible modules, equals to
that given by the the modulo $\varpi_{L_n}$ reduction of 
$H^k(\sh_{I^v(n),\bar s_v},\lexp p j^{=(r-k)g}_{!*}
HT_{\OC_{L_n},\xi'}(\pi'_v,\mathds 1_{r-k}))_{\mathfrak m}$.\footnote{More generally, lemma
\ref{lem-torsion-min} tells us that, in the following, we can argue as if the various cohomology groups
were torsion free.}

\item From \ref{prop-hip} and the description in \ref{defi-m} of the $m_{s,t}(r,i)$,
the behavior of the contribution to the sequence $d_{k,n}(\pi_v)-d_{k+1,n}(\pi_v)$ 
of the modulo $l$ reduction 
of the free quotients of each cohomology groups is completely determined by
$\BC_\xi(\pi_v,r_{\max})$. Indeed note that the contribution of some automorphic representation
$\Pi$ such that $\Pi_v \simeq \speh_{s}(\st_t(\pi_v)) \times ?$ with $s+t-1=r$,
only contributes to $d_{k,n}(\pi_v)$ for $k=0,\cdots, s-1$ so that 
$d_{s-1,n}(\pi_v)-d_{s,n}(\pi_v)$,
after eliminating the contribution of the torsion part which does not interfere thanks to lemma
\ref{lem-torsion-min}, will detect such $\Pi$.
\end{itemize}
\noindent
- We can then infer the elements
$(s,t) \in \BC_\xi(\pi_v,r_{\max})$ (resp. $\BC_{\xi'}(\pi'_v,r_{\max})$) for $s \geq 2$ and 
prove the quantitative
previous theorem for $s \geq 2$. 
\\
- Concerning the contribution of the elements of $\AC_{\xi,\pi_v}(r_{\max},1)$, we look
at $(d_{0,n}(\pi_v))_{n \in \Nm}$.
As we have seen that the contributions to the free quotients of the elements of
$\AC_{\xi,\pi_v}(r_{\max},s)$ for $s \geq 2$ coincide to the contributions to the free quotients
of the elements of $\AC_{\xi',\pi'_v}(r_{\max},s)$, we obtain that the remaining part given
by the elements of $\AC_{\xi,\pi_v}(r_{\max},1)$ coincides with those of $\AC_{\xi',\pi'_v}(r_{\max},1)$.

\medskip

At the end of this initialization step we proved that the contribution of the elements of
$\bigcup_{s \geq 1} \AC_{\xi,\pi_v}(r_{\max},s)$
coincide with those of the elements of $\bigcup_{s \geq 1} \AC_{\xi',\pi'_v}(r_{\max},s)$. 
We then can forget the contribution of these elements, in other words, we can now
suppose that the new $r_{\max}$ is strictly less than the real one and we can resume
the arguments.

More precisely, we argue by induction on $r$ from $r=r_{\max}$ to $r=1$ by assuming 
the statement of the theorem is true for every $r \geq k$ which is equivalent to say that
\begin{itemize}
\item for every $i,j$, the contribution of 
$\AC_{\xi,\pi_v}(k,s)$for $r \leq k \leq r_{\max}$ and $s \geq 1$
to the space given, after extending the scalar to $\overline \Fm_l$ and taking the inductive limit over
$n$, by modulo $\varpi_{L_n}$ reduction of the free quotient of $H^i(\sh_{I^v(n),\bar s_v},
\lexp p j^{=j g}_{!*} HT_{\OC_{L_n},\xi}(\pi_v,\Pi_j))_{\mathfrak m}$ 
\item is equal to those of $\AC_{\xi',\pi'_v}(k,s)$
to the space constructed from the modulo $\varpi_{L_n}$ reduction of the free quotient of 
$$H^i(\sh_{I^v(n),\bar s_v},
\lexp p j^{=j g}_{!*} HT_{\OC_{L_n},\xi'}(\pi'_v,\Pi_j))_{\mathfrak m}.$$
\end{itemize}

\rem Recall that lemma \ref{lem-torsion-min} tells us that we can argue as if all cohomological
groups were torsion free. Note also that in the range $r \leq i+k \leq r_{\max}$, in 
the free quotient of $H^i(\sh_{I^v(\oo),\bar s_v},\lexp p j^{=kg}_{!*}
HT_\xi(\pi_v,\Pi_k))_{\mathfrak m}$, only contributes elements
of $\AC_{\xi,\pi_v}(k,s)$ for $k\geq r$.

We can then forget about these elements and resume the previous arguments pretending that
we are in the case where $r_{\max}=r-1$ to conclude that the result is true for $r-1$.
\end{proof}

\section{Proof of conjecture 5.10 of \cite{boyer-aif}}

\begin{thm} 
Let $\varrho$ be any irreducible cuspidal representation of $GL_g(F_v)$
with $1 \leq g \leq d$.
For every irreducible cuspidal representation $\pi_v$ with $r_l(\pi_v) \simeq \varrho$,
for every $1 \leq t \leq d/g$ and for every $i$
$$\lim_{\atop{\rightarrow}{n}} H_{\tor}^i(\sh_{I^v(n),\bar s_v},j^{=tg}_! HT_{\OC_{L_n},\xi}(\pi_v,\Pi_t)) 
\otimes_{\OC_{L_n}} \overline \Fm_l$$
depends only on the modulo $l$ reduction of $\pi_v$ and $\xi$.
\end{thm}

\begin{proof}
As explained before it suffices to prove that the modulo $l$ reduction of the free quotient
of the inductive limit of the $H^i(\sh_{I^v(n),\bar s_v},j^{=tg}_! HT_{\OC_{L_n},\xi}(\pi_v,\Pi_t)) \otimes_{\OC_{L_n}} \overline \Qm_l$ 
depends only on the modulo $l$ reduction of $\pi_v$ and $\xi$.

We proved this is the case for the intermediate extensions. Then the result follows from the
precise descriptions given in the propositions \ref{prop-hip} and \ref{prop-hic} which, in short, explain
that the cohomology groups of the extension by zero of Harris-Taylor local systems is the
same as those of their intermediate extensions except for the coefficient multiplicities
$m_{s,t}(r,i)$ and $n_{s,t}(r,i)$ which change where the various contributions appear.

\end{proof}

%
%
%
%

\bibliographystyle{plain}
\bibliography{bib-ok}

\def\cftil#1{\ifmmode\setbox7\hbox{$\accent"5E#1$}\else
  \setbox7\hbox{\accent"5E#1}\penalty 10000\relax\fi\raise 1\ht7
  \hbox{\lower1.15ex\hbox to 1\wd7{\hss\accent"7E\hss}}\penalty 10000
  \hskip-1\wd7\penalty 10000\box7} \def\cprime{$'$}
\begin{thebibliography}{1}

\bibitem{boyer-invent2}
P.~Boyer.
\newblock Monodromie du faisceau pervers des cycles \'evanescents de quelques
  vari\'et\'es de {S}himura simples.
\newblock {\em Invent. Math.}, 177(2):239--280, 2009.

\bibitem{boyer-compositio}
P.~Boyer.
\newblock Cohomologie des syst\`emes locaux de {H}arris-{T}aylor et
  applications.
\newblock {\em Compositio}, 146(2):367--403, 2010.

\bibitem{boyer-torsion}
P.~Boyer.
\newblock Filtrations de stratification de quelques vari\'et\'es de {S}himura
  simples.
\newblock {\em Bulletin de la SMF}, 142, fascicule 4:777--814, 2014.

\bibitem{boyer-aif}
P.~Boyer.
\newblock Congruences automorphes et torsion dans la cohomologie d'un syst\`eme
  local d'{H}arris-{T}aylor.
\newblock {\em Annales de l'Institut Fourier}, 65 no. 4:1669--1710, 2015.

\bibitem{boyer-duke}
P.~Boyer.
\newblock La cohomologie des espaces de {L}ubin-{T}ate est libre.
\newblock {\em Duke Math. Journal}, pages 1531--1622, june 2023.

\bibitem{h-t}
M.~Harris, R.~Taylor.
\newblock {\em The geometry and cohomology of some simple {S}himura varieties},
  volume 151 of {\em Annals of Mathematics Studies}.
\newblock Princeton University Press, Princeton, NJ, 2001.

\bibitem{ito2}
T.~Ito.
\newblock Hasse invariants for somme unitary {S}himura varieties.
\newblock {\em Math. Forsch. Oberwolfach report 28/2005}, pages 1565--1568,
  2005.

\bibitem{zelevinski2}
A.~V. Zelevinsky.
\newblock Induced representations of reductive {${p}$}-adic groups. {II}. {O}n
  irreducible representations of {${\rm GL}(n)$}.
\newblock {\em Ann. Sci. \'Ecole Norm. Sup. (4)}, 13(2):165--210, 1980.

\end{thebibliography}
\end{document}